\documentclass[12pt]{article}
\usepackage{mathrsfs}
\usepackage{amsmath}
\usepackage{amsmath,amsthm,amssymb,amscd}
\usepackage{latexsym}
\usepackage[colorlinks,linkcolor=blue,anchorcolor=blue,citecolor=blue,CJKbookmarks=True]{hyperref}
\usepackage[numbers,sort&compress]{natbib}
\usepackage{cases}
\usepackage{makecell}
\usepackage{multirow}
\usepackage{geometry}
\usepackage{tabularx}
\usepackage{graphicx}
\usepackage{tikz}
\usepackage{tikz-cd}
\usepackage{titlesec}

\titleformat{\section}[block]
{\bfseries\large}
{\thesection\quad}
{0em}
{}

\allowdisplaybreaks

\newtheorem{theorem}{Theorem}[section]
\newtheorem{corollary}[theorem]{Corollary}
\newtheorem{lemma}[theorem]{Lemma}
\newtheorem{example}[theorem]{Example}
\newtheorem{proposition}[theorem]{Proposition}

\theoremstyle{definition}
\newtheorem{definition}[theorem]{Definition}
\newtheorem{remark}[theorem]{Remark}

\numberwithin{equation}{section}

\begin{document}
\begin{center}
{\large  \bf Notes on Reversed Products of Two Elements in Rings}\\
\vspace{0.8cm}   Long Wang\footnote{Email: lwangmath@yzu.edu.cn}, Tingting Li and Yuheng Liu \\
\vspace{0.5cm} {\small School of Mathematical Sciences, Yangzhou University, Yangzhou,  China}
\end{center}

\bigskip

{ \bf  Abstract:}  \leftskip0truemm\rightskip0truemm
This paper investigates reversed product properties of two elements in rings.
Motivated by Cline's formula,
we characterize reversible and $\ast$-reversible rings in terms of group invertible elements, EP elements,
and the transfer behaviour of generalized inverses for reversed products.
We prove that a unital ring $R$ is reversible if and only if $ab\in R^{\sharp}$ yields $ba\in R^{\sharp}$.
For an involutive ring $R$, $R$ is $\ast$-reversible precisely whenever $ab\in R^{\mathrm{EP}}$ implies $b^{\ast}a\in R^{\mathrm{EP}}$.
Several counterexamples are constructed to differentiate these ring classes,
and the mutual inclusion relations among them are also discussed.
 \bigskip

\section{Introduction}

The theory of generalized inverses constitutes an essential branch of noncommutative ring-theoretic research,
with wide-ranging applications in matrix analysis, operator algebras and other related fields.
Among all types of generalized inverses,
the Drazin inverse and the Moore-Penrose inverse are two of the most essential and well-studied notions.
As a renowned theorem for Drazin inverses, Cline's formula reveals the deep-seated connection
between the generalized invertibility of the products $ab$ and $ba$ over arbitrary rings \cite{C1}.
This celebrated result has inspired abundant investigations devoted to the transfer behaviour
of different generalized inverses for reversed products,
covering Drazin inverses, generalized Drazin inverses and Moore-Penrose inverses \cite{LZ1,LCC1,Ma1,MZ1,Mo1,SCLZ1}.

Most existing literature focuses on element-level questions.
Given two fixed ring elements $a,b$,
scholars aim to find necessary and sufficient conditions under which the generalized
invertibility of $ab$ passes over to $ba$.
These studies regard such transfer property as a local feature of separate ring elements.
Different from previous studies, the present work shifts
the research perspective from local element properties to ring-structural characteristics.
Rather than merely discussing the inheritance of generalized invertibility between reversed products,
we employ the transfer laws of group-invertible elements and EP elements as tools to characterize ring structures.

Reversible rings constitute a fundamental class of rings defined by zero-product conditions \cite{Co1},
which are closely related to annihilator conditions, idempotent elements,
Dedekind finiteness, and Abelian properties \cite{AC1,FN1,KL1,M1}.
When the ring is endowed with an involution, the classical reversible ring concept is naturally
extended to $\ast$-reversible rings \cite{FN1},
which are well adapted to the research framework of Moore-Penrose inverses in involutive rings.

Inspired by the Cline-type transfer rules for reversed product pairs,
we establish new equivalent characterizations for reversible rings and $\ast$-reversible rings in this paper.
We prove that a unital ring $R$ is reversible if and only if group invertibility
transfers from $ab$ to $ba$ for all $a,b\in R$. For a $\ast$-ring $R$,
the $\ast$-reversible property can be characterized via the transfer behavior of
EP elements: $R$ is $\ast$-reversible if and only if the EP property
transfers from $ab$ to $b^{\ast}a$ for all $a,b\in R$.

Moreover, we construct several concrete counterexamples to distinguish reversible rings,
$\ast$-reversible rings, and other closely related ring classes,
and further clarify the inclusion relations among these ring families.
Notably, the core contribution of this paper is the structural characterization of rings,
rather than a trivial discussion on the inheritance of invertibility or commutativity for element products.

\section{Preliminaries}
Throughout this paper, all rings are associative with unity.
In this section, we recall some basic definitions and notations required throughout the paper.

An element $a\in R$ is Drazin invertible \cite{Dr1} if there exists $x\in R$ such that
\[
a^{k}xa=a^{k},\quad xax=x,\quad ax=xa
\]
for some positive integer $k$. The minimal such integer $k$ is called the Drazin index of $a$,
and the unique element $x$ is termed the Drazin inverse of $a$, denoted by $a^{D}$.
In particular, $a$ admits a group inverse if and only if its Drazin index equals $1$.
The group inverse is uniquely determined and denoted by $a^{\sharp}$.
We write $R^{\sharp}$ for the set of all group-invertible elements of $R$.

An involution on a ring $R$ is an order-$2$ anti-isomorphism $a\mapsto a^{\ast}$ satisfying
\[
(a^{\ast})^{\ast}=a,\quad (a+b)^{\ast}=a^{\ast}+b^{\ast},\quad (ab)^{\ast}=b^{\ast}a^{\ast}.
\]
A ring equipped with an involution is called a $\ast$-ring.

An element $a\in R$ of a $\ast$-ring is Moore-Penrose invertible \cite{PT1} if there exists $x\in R$ such that
\[
axa=a,\quad xax=x,\quad (ax)^{\ast}=ax,\quad (xa)^{\ast}=xa.
\]
The unique such element $x$ is the Moore-Penrose inverse of $a$,
denoted by $a^{\dag}$. We denote by $R^{\dag}$ the set of all Moore-Penrose invertible elements of $R$.

Let $R$ be a $\ast$-ring. An element $a\in R$ is said to be an EP element
if $a\in R^{\sharp}\cap R^{\dag}$ and $a^{\sharp}=a^{\dag}$.
The set of all EP elements of $R$ is denoted by $R^{\mathrm{EP}}$.
An element $p\in R$ is called a projection if $p^{2}=p=p^{\ast}$.
We write $P(R)$ for the set of all projections of $R$.

\section{Characterizations of Reversible Rings by Group Invertibility}
We begin with the definition and fundamental properties of reversible rings.
Let \(R\) be a unital ring.
A ring \(R\) is reversible \cite{Co1} if \(ab=0\) implies \(ba=0\) for all \(a,b\in R\),
which is equivalent to the condition that the left annihilator \(l(b)\)
equals the right annihilator \(r(b)\) for each \(b\in R\).

Let \(E(R)\) denote the set of all idempotent elements of \(R\).
It was shown in \cite{CWZ1} that a ring \(R\) is reversible
if and only if \(ab\in E(R)\) implies \(ba\in E(R)\),
and this condition is further equivalent to \(ab\in E(R)\) yielding \(ab=ba\).
Furthermore, for a \(\ast\)-ring \(R\), \(R\) is reversible
if and only if \(ab\in P(R)\) implies \(ba\in P(R)\),
or equivalently, \(ab\in P(R)\) implies \(ab=ba\) \cite{CWZ1}.

Based on the invertibility transfer of ring elements,
we further consider Dedekind-finite rings,
another essential class of rings closely related to product invertibility.
Let \(U(R)\) stand for the set of all invertible units of the ring \(R\).
A ring \(R\) is Dedekind-finite if \(ab=1\) implies \(ba=1\) for all \(a,b\in R\).
A natural question arises: for arbitrary elements \(a,b\in R\),
whether \(ab\in U(R)\) can guarantee \(ba\in U(R)\),
and what kinds of rings satisfy such product invertibility transfer property.

\begin{example}
Let $R$ be the ring of infinite matrices over the real number field $\mathbb{R}$, which are
row and column-finite.
Choose
\begin{eqnarray*}
a=\left(
     \begin{array}{cccc}
       0 & 1   &   &   \\
        & 0 & 1   &   \\
         &  & 0 & \ddots  \\
         &  &  & \ddots \\
     \end{array}
   \right), \ \ \
b=\left(
     \begin{array}{cccc}
       0 &   &   &   \\
       1 & 0 &   &   \\
         & 1 & 0 &   \\
         &   & \ddots & \ddots \\
     \end{array}
   \right).
\end{eqnarray*}
Then we compute:
\begin{eqnarray*}
ab=\left(
     \begin{array}{cccc}
       1 &   &   &   \\
         & 1 &   &   \\
         &   & 1 &   \\
         &   &   & \ddots \\
     \end{array}
   \right) \in U(R), \ \ \
ba=\left(
     \begin{array}{cccc}
       0 &   &   &   \\
         & 1 &   &   \\
         &   & 1 &   \\
         &   &   & \ddots \\
     \end{array}
   \right) \not\in U(R).
\end{eqnarray*}
\end{example}

\begin{proposition}\label{prop2.01}
Let $R$ be a ring with unity, and let $a,b\in R$. The following statements are equivalent:

(1) $R$ is Dedekind-finite

(2) For all $a,b\in R$, $ab\in U(R)$ implies $ba\in U(R)$.
\end{proposition}

\begin{proof}
$(1) \Rightarrow (2)$: Suppose $ab\in U(R)$. Then there exists an element $u\in R$ satisfying $uab=1$ and $abu=1$.
Since $R$ is Dedekind-finite, the equality $uab=1$ implies $bua=1$.
This shows that $b\in U(R)$ with $b^{-1}=ua$.
Similarly, we deduce that $a\in U(R)$ and $a^{-1}=bu$.
Hence, we can see that $ba\in U(R)$, and its inverse is given by $(ba)^{-1}=a^{-1}b^{-1}$.

$(2) \Rightarrow (1)$: Assume $ab=1$. Then $ab\in U(R)$, so condition (2) guarantees $ba\in U(R)$.
Let $u\in R$ be the inverse of $ba$, so $bau=1$ and $uba=1$.
Combining $ab=1$ and $bau=1$, we have then $b \in U(R)$ and $b^{-1}=a=au$.
Similarly, we can see $a \in U(R)$. Multiplying both sides of $a=au$ by $a^{-1}$ on the left yields $u=1$.
Substituting $u=1$ into $bau=1$, we obtain $ba=1$.
Hence $R$ is Dedekind-finite.
\end{proof}

\begin{remark}
We emphasize that the equality $ab=ba$ does not hold in general under the hypotheses of Proposition \ref{prop2.01}.
It is well known that the full $3\times 3$ matrix ring over the complex field $M_3(\mathbb{C})$ is a Dedekind-finite ring.
Consider the matrices
\[
A=
\begin{pmatrix}
1 & 0 & 1 \\
0 & 1 & 0 \\
0 & 0 & 1
\end{pmatrix},\qquad
B=
\begin{pmatrix}
1 & 0 & 0 \\
0 & 1 & 0 \\
0 & -1 & 1
\end{pmatrix}.
\]
Direct computation shows that both $AB$ and $BA$ are invertible elements of $M_3(\mathbb{C})$, while $AB\neq BA$.
\end{remark}

As a generalization of invertibility, it is natural to consider generalized inverses.
For the problem concerning reversed products of two elements,
the theory of generalized inverses manifests itself precisely as the Cline's formula \cite{C1}.
It is well known that if $ab \in R^{D}$, then $ba\in R^{D}$, with the representation
$$(ba)^{D}=b((ab)^{D})^{2}a.$$

We emphasize here that $ab\in R^{\sharp}$ only guarantees $ba\in R^{D}$, but does not entail $ba \in R^{\sharp}$.
Accordingly, a natural question arises: which rings satisfy the condition
that $ab\in R^\sharp$ implies $ba\in R^\sharp$ for all $a,b\in R$?

\begin{theorem}\label{thm2.04}
Let $R$ be a ring with unity, and let $a,b\in R$. The following statements are equivalent:

(1) $R$ is reversible.

(2) $ab\in R^{\sharp}$ implies $ba\in R^{\sharp}$.
\end{theorem}

\begin{proof}
$(1)\Rightarrow(2)$: Suppose $ab\in R^\sharp$. By the Cline's formula, $ba\in R^D$ and the Drazin inverse satisfies the identity
\[
(ba)^D = b\big((ab)^\sharp\big)^2 a.
\]
To verify $ba\in R^\sharp$, it suffices to check the group invertibility condition $ba = (ba)^2(ba)^D$.
Since $ab$ is group invertible, we have $a\big(1 - b(ab)^\sharp a\big)b = 0$.
As $R$ is reversible, $ba\big(1 - b(ab)^\sharp a\big) = 0$, which yields $ba = bab(ab)^\sharp a$.
Substituting the expression for $(ba)^D$, we compute
\[
(ba)^2(ba)^D = (ba)^2 b\big((ab)^\sharp\big)^2 a = bab(ab)^\sharp a = ba.
\]
Therefore, $ba\in R^\sharp$.

$(2)\Rightarrow(1)$: Assume $ab=0$. Trivially, $ab\in R^\sharp$, so condition (2) forces $ba\in R^\sharp$.
Then $ba=(ba)^{2}(ba)^{\sharp}=0$. Hence, $R$ is reversible.
\end{proof}

It is well-known that a ring $R$ is strongly regular if and only if every element in $R$ is group invertible.
Hence, the following corollary naturally follows from Theorem \ref{thm2.04}.

\begin{corollary}
Every strongly regular ring is reversible.
\end{corollary}

\begin{remark}
It is known that a ring $R$ is reversible if and only if $ab\in E(R)$ implies $ba\in E(R)$.
This motivates us to consider the following equivalent characterization concerning reversible rings.

Let $R$ be a ring and let $a,b\in R$. Consider the following statements:
\begin{enumerate}
\item[(1)] $R$ is reversible.
\item[(2)] $ab\in E(R)$ implies $ba\in R^{\sharp}$.
\item[(3)] $ab\in R^{\sharp}$ implies $ba\in E(R)$.
\end{enumerate}

We now verify the relations among these three conditions.

$(1)\Rightarrow(2)$ Suppose that $ab\in E(R)$. Then $ab\in R^{\sharp}$. By Theorem \ref{thm2.04}, we have $ba\in R^{\sharp}$.

$(2)\Rightarrow(1)$ If $ab=0$, then $ab\in E(R)$. Condition $(2)$ yields $ba\in R^{\sharp}$.
It follows that $ba = b(ab)a(ba)^{\sharp}=0$, so $R$ is reversible.

$(3)\Rightarrow(1)$ Let $ab\in R^{\sharp}$. By condition $(3)$, we obtain $ba\in E(R)$, which implies $ba\in R^{\sharp}$.
Thus $R$ is reversible by Theorem \ref{thm2.04}.

Nevertheless, the implication $(1)\Rightarrow(3)$ fails in general.
In fact, it is not difficult to prove that condition $(3)$ is equivalent to $R^{\sharp}\subseteq E(R)$.
We illustrate this by two examples.

First, a strongly regular ring may not satisfy $(3)$.
Take $R=\mathbb{Z}_{3}$. The element $[2]\in R^{\sharp}$, but $[2]\notin E(R)$.
This example also demonstrates that the implication $(1)\Rightarrow(3)$ does not hold in general.

Second, a ring satisfying $(3)$ need not be strongly regular.
For $R=\mathbb{F}_{2}$, we have $R^{\sharp}=E(R)=\{0,1\}$, so condition $(3)$ holds,
while $\mathbb{F}_{2}$ is not strongly regular. Indeed, $x$ is group invertible.
\end{remark}

\begin{remark}
(1) We remark that the equality $ab=ba$ does not generally hold under the setting of Theorem \ref{thm2.04}.
Indeed, non-commutative strongly regular rings exist.
Let $C(R)$ stand for the centre of $R$.
Nevertheless, we claim that $ab=ba$ whenever $ab\in C(R)$.

Since $ba\in R^{\sharp}$, the idempotent $ba(ba)^{\sharp}$ belongs to $C(R)$, because every reversible ring is Abelian.
Direct computation yields
\[
\begin{aligned}
ba &= b(ab)a(ba)^{\sharp}=(ab)ba(ba)^{\sharp}=(ab)^{\sharp}abab\bigl[ba(ba)^{\sharp}\bigr] \\
&=(ab)^{\sharp}aba\bigl[ba(ba)^{\sharp}\bigr]b=(ab)^{\sharp}abab=ab.
\end{aligned}
\]

(2) As shown in \cite{N1}, any group invertible element $r$ admits the decomposition $r=ue=eu$, where $u\in U(R)$ and $e\in E(R)$.
Accordingly, condition (2) can be restated as:
$abe_{1}\in R^{\sharp}$ implies $bae_{2}\in R^{\sharp}$ for certain idempotents $e_{1}, e_{2} \in E(R)$.
This research direction will be presented in the following sections.
\end{remark}

In \cite{CV1}, the author introduced the notions of $e$-reversible rings and strongly $e$-reversible rings via an idempotent element $e$, and investigated their basic properties together with the connections between these rings, reversible rings and symmetric rings.
In particular, it was proved that the idempotent $e$ is left semicentral whenever $R$ is an $e$-reversible ring.
Recall that an idempotent $e$ is said to be left semicentral if $(1-e)re=0$ holds for all $r\in R$.
Moreover, the idempotent $e\in C(R)$ provided that $R$ is strongly $e$-reversible.
The author also established the equivalence between $e$-reversible rings and strongly $e$-reversible rings: $R$ is strongly $e$-reversible if and only if $R$ is $e$-reversible and $e\in C(R)$.

In the subsequent part of this section, we present characterizations for $e$-reversible and strongly $e$-reversible rings by virtue of the transfer property for reversed products of group invertible elements and idempotents.

\begin{definition}
Let $R$ be a ring and $e\in E(R)$. Then $R$ is called $e$-reversible \cite{CV1}
if for all $a,b\in R$, $ab = 0$ implies $bae = 0$.
\end{definition}

\begin{lemma}\label{lem2.09}
{\rm(\cite{CV1}, Theorem 2.4)} If $R$ is $e$-reversible, then $e$ is left semicentral.
\end{lemma}

In what follows, we give several characterizations of $e$-reversible rings.

\begin{theorem}\label{thm2.10}
Let $R$ be a ring, and let $e\in E(R)$. The following statements are equivalent:

(1) $R$ is $e$-reversible.

(2) $abe=0$ implies $bae\in E(R)$.

(3) $abe=0$ implies $bae\in R^{\sharp}$.

(4) $abe=0$ implies $beae=aebe$.

(5) $ab=0$ implies $bae\in E(R)$.

(6) $ab=0$ implies $bae\in R^{\sharp}$.

(7) $ab=0$ implies $beae=aebe$.
\end{theorem}

\begin{proof}
We first establish a preliminary claim: each of conditions $(2)\sim(7)$ implies that the idempotent $e$ is left semicentral.

To justify this claim, set $a=e$ and $b=(1-e)xe$. Then $ab=abe=0$, which gives $bae=beae=(1-e)xe$. We consider three cases:
\[
\begin{cases}
\text{If } bae\in R^{\sharp}, & \text{then } (1-e)xe=\bigl[(1-e)xe\bigr]^2\bigl[(1-e)xe\bigr]^{\sharp}=0;\\[4pt]
\text{If } bae\in E(R), & \text{then } (1-e)xe=\bigl[(1-e)xe\bigr]^2=0;\\[4pt]
\text{If } beae=aebe, & \text{then } (1-e)xe=e(1-e)xe=0.
\end{cases}
\]

$(1)\Rightarrow(2)$ Trivial. Suppose $a(be)=0$. The $e$-reversibility of $R$ yields $(be)ae=0$.
By Lemma \ref{lem2.09}, $e$ is left semicentral. Thus $bae=beae=0\in E(R)$.

$(2)\Rightarrow(3)$ This implication holds trivially since $E(R)\subseteq R^{\sharp}$.

$(3)\Rightarrow(4)$ Let $abe=0$. Condition $(3)$ ensures $bae\in R^{\sharp}$.
As $e$ is left semicentral, we compute
\[
bae=baebae(bae)^{\sharp}=babeae(bae)^{\sharp}=0.
\]
Moreover, $beae=bae=0$ and $aebe=abe=0$, so $beae=aebe$.

$(4)\Rightarrow(5)$ If $ab=0$, then $abe=0$. Applying condition $(4)$, we obtain $beae=aebe$.
Recalling that $e$ is left semicentral, we have
\[
bae=beae=aebe=abe=0\in E(R).
\]

The implications $(5)\Rightarrow(6)$ and $(6)\Rightarrow(7)$ follow by arguments analogous to those for $(2)\Rightarrow(3)$ and $(3)\Rightarrow(4)$, respectively.

$(7)\Rightarrow(1)$ Assume $ab=0$. Then $beae=aebe$ by condition $(7)$.
Since $e$ is left semicentral, it follows that
\[
bae=beae=aebe=abe=0,
\]
which verifies that $R$ is $e$-reversible.
\end{proof}

Next, we replace the zero-product conditions $ab=0$ and $abe=0$ with idempotent conditions.
Specifically, we investigate the settings where $ab\in E(R)$ and $abe\in E(R)$.
Before stating our main results, we first prove an auxiliary lemma for $e$-reversible rings.

\begin{lemma}\label{lem2.11}
If $R$ is $e$-reversible, then $fre = rfe$ for all $f\in E(R)$ and $r\in R$.
\end{lemma}

\begin{proof}
Let $f\in E(R)$ and $r\in R$. First set $a=f$ and $b=(1-f)rf$. Then $ab=0$, and the $e$-reversibility of $R$ implies $bae=0$, i.e., $(1-f)rfe=0$.
Next, take $a=fr(1-f)$ and $b=f$. Then $ab=0$, so again $bae=0$, which yields $fr(1-f)e=0$.
Combining the two equalities above, we obtain $fre = rfe = frfe$.
\end{proof}

\begin{proposition}\label{prop2.12}
Let $R$ be a ring, and let $e\in E(R)$. The following statements are equivalent:

(1) $R$ is $e$-reversible.

(2) $ab \in E(R)$ implies $bae\in E(R)$.

(3) $ab \in E(R)$ implies $bae\in R^{\sharp}$.

(4) $ab \in E(R)$ implies $bae=abe$.

(5) $abe \in E(R)$ implies $bae\in E(R)$.

(6) $abe \in E(R)$ implies $bae\in R^{\sharp}$.

(7) $abe \in E(R)$ implies $beae=aebe$.
\end{proposition}
\begin{proof}
Similar to the proof of Theorem \ref{thm2.10}, each of conditions $(2)\sim (7)$ forces $e$ to be left semicentral.

$(1)\Rightarrow(2)$ Suppose $ab\in E(R)$. Then $(1-ab)ab=0$. Applying condition (1), we obtain $b(1-ab)ae=0$.
Since $e$ is left semicentral, it follows that $bae=babae=(bae)^{2}$.

$(2)\Rightarrow(3)$ This implication holds trivially since $E(R)\subseteq R^{\sharp}$.

$(3)\Rightarrow(1)$ If $ab=0$, then $ab\in E(R)$. By condition (3), $bae\in R^{\sharp}$. Therefore,
\[
bae=(bae)^{2}(bae)^{\sharp}=babae(bae)^{\sharp}=0.
\]

$(1)\Rightarrow(4)$ Let $ab\in E(R)$.
Note that we have already established the equivalence $(1)\Leftrightarrow(2)$, so $bae\in E(R)$ whenever $ab\in E(R)$.
Combined with Lemma \ref{lem2.11}, we deduce
\[
bae=babae=(ab)bae=a(bae)be=ababe=abe.
\]

$(4)\Rightarrow(5)$ Assume $abe \in E(R)$. Condition (4) yields $beae=abe$, which implies $bae=abe\in E(R)$.

$(5)\Rightarrow(6)$ This is immediate since $E(R) \subseteq R^{\sharp}$.

$(6)\Rightarrow(1)$ If $ab=0$, then $abe=0\in E(R)$.
By (6), $bae\in R^{\sharp}$, and hence
\[
bae=(bae)^{2}(bae)^{\sharp}=babae(bae)^{\sharp}=0.
\]

$(1)\Rightarrow(7)$ Suppose $abe\in E(R)$. Note that we have verified $(1)\Leftrightarrow(5)$, so $bae\in E(R)$.
Hence, we can see that
\[
beae=bae=baebae=b(abe)ae=abe(bae)=a(bae)be=abeabe=abe=aebe.
\]

$(7)\Rightarrow(1)$ Let $ab=0$. Then $abe=0\in E(R)$. Condition (7) gives
\[
bae=beae=aebe=abe=0,
\]
which shows that $R$ is $e$-reversible.
\end{proof}

\begin{proposition}\label{prop2.13}
Let $R$ be a ring, and let $e\in E(R)$. The following statements are equivalent:

(1) $R$ is $e$-reversible.

(2) $ab\in R^{\sharp}$ implies $bae\in R^{\sharp}$.

(3) $abe\in R^{\sharp}$ implies $bae\in R^{\sharp}$.

(4) $abe\in R^{\sharp}$ implies $beae\in R^{\sharp}$.
\end{proposition}
\begin{proof}
Similar to the proof of Theorem \ref{thm2.10}, each of conditions $(2)\sim(4)$ implies that $e$ is left semicentral.
It suffices to set $a=e$ and $b=(1-e)xe$.
Then $(1-e)xe\in R^{\sharp}$, which yields
$(1-e)xe=\bigl[(1-e)xe\bigr]^{2}\bigl[(1-e)xe\bigr]^{\sharp}=0$.

$(1)\Rightarrow(2)$ Suppose $ab\in R^{\sharp}$. Then $ab\bigl[1-ab(ab)^{\sharp}\bigr]=0$. Applying condition (1), we obtain
\[
b\bigl[1-ab(ab)^{\sharp}\bigr]ae=0,
\]
which simplifies to $bae=bab(ab)^{\sharp}ae$. We further derive two inclusions:
\[
bae=babab\bigl[(ab)^{\sharp}\bigr]^2ae=baebaeb\bigl[(ab)^{\sharp}\bigr]^2ae\in (bae)^2R,
\]
and similarly
\[
bae=bab(ab)^{\sharp}ae=bab\bigl[(ab)^{\sharp}\bigr]^3ababae=bab\bigl[(ab)^{\sharp}\bigr]^3abaebae\in R(bae)^2.
\]
These two containments jointly imply $bae\in R^{\sharp}$.

$(2)\Rightarrow(3)$ Assume $abe\in R^{\sharp}$. By condition $(2)$, we have $be\cdot a\cdot e\in R^{\sharp}$.
Since $e$ is left semicentral, it follows that $beae=bae$, so $bae\in R^{\sharp}$.

$(3)\Rightarrow(4)$ This implication is trivial. Since $e$ is left semicentral, we immediately have $bae=beae$.

$(4)\Rightarrow(1)$ Let $ab=0$. Then $abe=0\in R^{\sharp}$. Condition $(4)$ guarantees $bae\in R^{\sharp}$. We compute
\[
bae=(bae)^2(bae)^{\sharp}=babae(bae)^{\sharp}=0.
\]
Hence, $R$ is $e$-reversible, which completes the proof.
\end{proof}

Similarly, the notion of strongly $e$-reversible rings is introduced for central idempotents $e\in C(R)$.

\begin{definition}
Let $R$ be a ring and $e\in E(R)$. Then $R$ is said to be strongly $e$-reversible \cite{CV1}
if for all $a,b\in R$, $ab = 0$ implies $bea = 0$.
\end{definition}

\begin{lemma}\label{lem2.15}
{\rm(\cite{CV1}, Corollary 2.6)} A ring $R$ is strongly $e$-reversible if and only if $R$ is $e$-reversible and $e\in C(R)$.
\end{lemma}

From Lemma \ref{lem2.15}, we obtain the following characterization of strongly e-reversible rings. The proofs are analogous to those of Theorem \ref{thm2.10}, Proposition \ref{prop2.12} and Proposition \ref{prop2.13}, and we leave the details to the reader.

\begin{proposition}\label{prop2.16}
Let $R$ be a ring, and let $e\in E(R)$. The following statements are equivalent:

(1) $R$ is strongly $e$-reversible.

(2) $ab=0$ implies $bea\in E(R)$.

(3) $ab=0$ implies $bea\in R^{\sharp}$.

(4) $ab=0$ implies $aeb=bea$.

(5) $ab \in E(R)$ implies $bea\in E(R)$.

(6) $ab \in E(R)$ implies $bea\in R^{\sharp}$.

(7) $ab \in E(R)$ implies $bea=aeb$.

(8) $ab \in R^{\sharp}$ implies $bea\in R^{\sharp}$.
\end{proposition}

\section{Characterizations of $\ast$-Reversible Rings by EP Elements}

In the previous section, we studied characterizations of reversible rings in terms of group-invertible elements.
Recall that a ring $R$ is reversible if and only if $ab\in R^{\sharp}$ implies $ba\in R^{\sharp}$ for all $a,b\in R$.
In the present section,
it is natural to ask how to characterize various generalized reversible rings
under the condition that elements are Moore-Penrose invertible.
Since Moore-Penrose invertibility is intrinsically defined for involutive rings,
we are naturally led to consider the corresponding notion of reversible rings, namely $\ast$-reversible rings.
Before proceeding with this topic, we first introduce some basic notation and auxiliary lemmas.

\begin{definition}\label{def3.01}
A $\ast$-ring $R$ is said to be $\ast$-reversible \cite{FN1} if $ab=0$ implies $b^{\ast}a=0$ for all $a, b\in R$.
\end{definition}

\begin{lemma}\label{lem3.02}
{\rm(\cite{FN1}, Proposition 6)}
Every $\ast$-reversible ring is reversible.
\end{lemma}

\begin{definition}\label{def3.03}
Let $R$ be a $\ast$-ring.
An element $r\in R$ is called an \textit{EP element}
if $r$ is both Moore-Penrose invertible and group invertible with $r^{\dagger}=r^{\sharp}$.
Equivalently, $r\in R$ is EP if $r$ is Moore-Penrose invertible and satisfies $rr^{\dagger}=r^{\dagger}r$.
\end{definition}
We denote by $R^{\mathrm{EP}}$ the set of all EP elements in $R$.
Regarding the connections between EP elements and ring structures,
we state the following conclusions, which were proved in our previous paper.
These results will be employed in the proofs of the subsequent theorems.

\begin{lemma}\label{lem3.04}
{\rm(\cite{WQW1}, Theorem 4.7)}
Let $R$ be a $\ast$-ring and let $a\in R$. Then $a$ is $*$-strongly regular if and only if $a$ is an EP element.
\end{lemma}

\begin{lemma}\label{lem3.05}
Let $R$ be a $\ast$-ring. If $a$ is $\ast$-strongly regular, then there exist a unit $u$ and a projection $p$ such that $a=up=pu$.
\end{lemma}

\begin{proof}
If $a$ $\ast$-strongly regular, then $a^{\sharp}$ and $a^{\dag}$ both exist and $a^{\sharp}=a^{\dag}$.
Set $p=aa^{\dag}$ and $u=a+1-p$.
Then $u\in U(R)$ and $u^{-1}=a^{\sharp}+1-p$.
Thus, $a=up=pu$.
\end{proof}

We now address the conjecture mentioned above:
\begin{center}$R$ is $\ast$-reversible if and only if $ab\in R^{\mathrm{EP}}$ implies $b^{*}a\in R^{\mathrm{EP}}$.\end{center}

At this point, some explanation is needed for our choice of EP elements in place of mere Moore-Penrose invertibility.
Indeed, we have the following result concerning $\ast$-reversible rings.

\begin{lemma}\label{lem3.06}
If $R$ is $*$-reversible, then every regular element of $R$ is an EP element.
\end{lemma}

\begin{proof}
Let $a\in R$ be a regular element.
Then there exists $x\in R$ satisfying $axa=a$.
We obtain $(1-ax)a=0$ and $a(1-xa)=0$.
Since $R$ is $\ast$-reversible, by Lemma \ref{lem3.02}, $(1-ax)a=0$ yields $a(1-ax)=0$.
Similarly, $a(1-xa)=0$ implies $(1-xa)a=0$.
Thus, $a=a^{2}x=xa^{2}$, and consequently,
\[
xa=xa^{2}x=ax.
\]
Put $y=xax$. Direct computation gives
\[
aya=axaxa=a,\qquad
yay=xaxaxax=xax=y,
\]
\[
ay=axax=ax=xa=xaxa=ya.
\]
Thus $a$ is group-invertible with group inverse $y$.

Now observe $(1-ay)ay=0$.
By the $*$-reversibility of $R$, $(ay)^{*}(1-ay)=0$,
which yields $(ay)^{*}=(ay)^{*}ay$.
Hence $ay$ is a self-adjoint idempotent.
Recall $ay=ya$, and one can further verify that $y$ serves as the Moore-Penrose inverse of $a$.
Therefore $a^{\sharp}=y=a^{\dagger}$, and $a$ is an EP element.
\end{proof}

\begin{theorem}\label{thm3.07}
Let $R$ be a $*$-ring. Then
$R$ is $*$-reversible if and only if $ab\in R^{\mathrm{EP}}$ implies $b^{*}a\in R^{\mathrm{EP}}$.
\end{theorem}

\begin{proof}
Proof of sufficiency. Suppose that $ab\in R^{\mathrm{EP}}$. In view of Lemma \ref{lem3.06}, it suffices to show that $b^{*}a$ is regular in $R$.

From the identity $\bigl(1-ab(ab)^{\sharp}\bigr)ab=0$, the $*$-reversibility of $R$ yields
\[
b^{*}\bigl(1-ab(ab)^{\sharp}\bigr)a=0.
\]
Expanding this equality, we obtain
\[
b^{*}a=b^{*}ab(ab)^{\sharp}a.
\]
Since $ab\in R^{\mathrm{EP}}$, the projection $ab(ab)^{\sharp}$ is self-adjoint, i.e. $\bigl(ab(ab)^{\sharp}\bigr)^{*}=ab(ab)^{\sharp}$.
Furthermore, every $*$-reversible ring is Abelian, i.e., $E(R) \subseteq C(R)$. Combining this fact, we obtain
\begin{align*}
b^{*}a=\bigl[(ab)^{\sharp}ab\bigr]^{*}\bigl[ab(ab)^{\sharp}\bigr]b^{*}a
=b^{*}a^{*}\bigl((ab)^{\sharp}\bigr)^{*}\bigl[ab(ab)^{\sharp}\bigr]b^{*}a
=b^{*}\bigl[ab(ab)^{\sharp}\bigr]a^{*}\bigl((ab)^{\sharp}\bigr)^{*}b^{*}a.
\end{align*}
This shows that $b^{\ast}a \in (b^{\ast}a)R(b^{\ast}a)$.
So, $b^{\ast}a$ is regular in $R$.

Proof of necessity.
We first claim that every projection in $R$ is central, i.e., $P(R)\subseteq C(R)$.
Take an arbitrary projection $p\in P(R)$ and arbitrary element $r\in R$.
Set $a=pr(1-p)$ and $b=p$.
Then $ab=0\in R^{\mathrm{EP}}$.
By hypothesis, $b^{*}a=pr(1-p)\in R^{\mathrm{EP}}$.
Consequently
\[
pr(1-p)=\bigl[pr(1-p)\bigr]^{2}\bigl[pr(1-p)\bigr]^{\sharp}=0.
\]
Similarly, put $a=(1-p)rp$ and $b=1-p$. Then $(1-p)rp=0$.
Combining the two equalities, we conclude $P(R)\subseteq C(R)$.

We now proceed to prove that $R$ is $*$-reversible.
Suppose $ab=0$. From our hypothesis, it follows that $b^{*}a\in R^{\mathrm{EP}}$.
By Lemma \ref{lem3.04} and Lemma \ref{lem3.05},
there exist $u\in U(R)$ and $p\in P(R)$ satisfying $b^{*}a=up=pu$.
Hence $p=u^{-1}b^{*}a=b^{*}au^{-1}$.
Since $p$ is a projection, $p\in C(R)$, one can see
\begin{align*}
p&=pp^{*}=p\bigl(b^{*}au^{-1}\bigr)^{*} =p(u^{-1})^{*}a^{*}b =(u^{-1})^{*}a^{*}pb=(u^{-1})^{*}a^{*}\bigl(u^{-1}b^{*}a\bigr)b=0.
\end{align*}
Thus $p=0$, which yields $b^{*}a=0$.
This completes the proof.
\end{proof}

It is well known that an element $r$ is EP if and only if $r^{*}$ is EP.
From Theorem \ref{thm3.07}, we derive the following result.

\begin{corollary}\label{cor3.08}
Let $R$ be a $\ast$-ring. Then $R$ is $*$-reversible if and only if $ab\in R^{\mathrm{EP}}$ implies $a^{*}b\in R^{\mathrm{EP}}$.
\end{corollary}

In the final section,
we return to the issue concerning Moore-Penrose inverses mentioned earlier and establish the following result.

\begin{proposition}\label{pro3.09}
Let $R$ be a $*$-ring. If $R$ is $*$-reversible,
then $ab\in R^{\dagger}$ implies $b^{*}a\in R^{\dagger}$.
\end{proposition}

\begin{proof}
Suppose that $ab\in R^{\dagger}$.
By Lemma \ref{lem3.06}, then $ab\in R^{\mathrm{EP}}$.
Since $R$ is $*$-reversible, Theorem \ref{thm3.07} gives $b^{*}a\in R^{\mathrm{EP}}$.
So we obtain $b^{*}a\in R^{\dagger}$.
\end{proof}

\begin{remark}\label{rem3.10}
The converse of Proposition \ref{pro3.09} does not hold in general.
We now demonstrate this fact.

Let $R=M_{2}(\mathbb{C})$ with the involution $*$ defined by conjugate transpose.
Since $M_{2}(\mathbb{C})$ is $*$-regular, every complex matrix is Moore-Penrose invertible.
Consequently, the condition $ab\in R^{\dagger}\implies b^{*}a\in R^{\dagger}$ holds trivially.
Take
\[
a=\begin{pmatrix}0&1\\0&0\end{pmatrix},\quad
b=\begin{pmatrix}1&0\\0&0\end{pmatrix}.
\]
Then $ab=0$, whereas $b^{*}a=ba=\begin{pmatrix}0&1\\0&0\end{pmatrix}\ne0$.
Hence $R$ is not $*$-reversible.
\end{remark}

\begin{remark}
Some rings satisfy $ab\in R^{\dagger}\implies b^{*}a\in R^{\dagger}$ yet are not $*$-regular.
Consider $R=\mathbb{Z}$ with the identity involution $n^{*}=n$.
Clearly $R^{\dagger}=\{-1,0,1\}$.
Direct computation verifies the above-mentioned condition.
Meanwhile, $\mathbb{Z}$ fails to be $*$-regular owing to $R\neq R^{\dagger}$.
\end{remark}

In fact, we can obtain the inclusion relations among these classes of rings, as illustrated in the diagram below.

\includegraphics[width=4.5in]{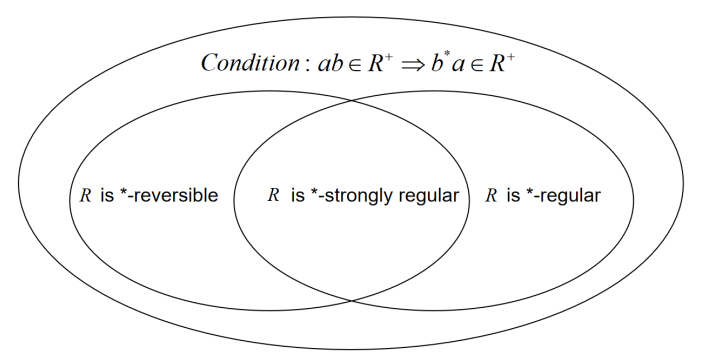}
\vspace{1mm}

Based on the characterizations of $\ast$-reversible rings via Moore-Penrose inverses,
we finally turn to the characterization of reversible rings under Moore-Penrose invertibility.
We first state an auxiliary lemma before presenting our main result.

\begin{lemma}\label{lem3.12}
Let $R$ be a $\ast$-ring. If $R$ is reversible,
then every Moore-Penrose invertible element of $R$ is an EP element.
\end{lemma}

\begin{proof}
Let $a\in R$ be a Moore-Penrose invertible element.
Then there exists $a^{\dag}\in R$ satisfying $aa^{\dag}a=a$.
Since $R$ is reversible,  $(1-aa^{\dag})a=0$ yields $a(1-aa^{\dag})=0$.
Similarly, $a(1-a^{\dag}a)=0$ implies $(1-a^{\dag}a)a=0$.
Thus, $a=a^{2}a^{\dag}=a^{\dag}a^{2}$, and consequently,
\[
a^{\dag}a=a^{\dag}a^{2}a^{\dag}=aa^{\dag}.
\]
Therefore $a$ is an EP element.
\end{proof}

\begin{proposition}\label{pro3.13}
Let $R$ be a $\ast$-ring. If $R$ is reversible,
then $ab\in R^{\dagger}$ implies $ba\in R^{\dagger}$.
\end{proposition}

\begin{proof}
Suppose that $ab\in R^{\dagger}$.
By Lemma \ref{lem3.12}, $ab\in R^{\mathrm{EP}}$.
Since $R$ is reversible, Theorem \ref{thm2.04} yields $ba\in R^{\sharp}$.

From the identity $(1-ab(ab)^{\sharp})ab=0$, the reversibility of $R$ gives $b\bigl(1-ab(ab)^{\sharp}\bigr)a=0$,
which simplifies to $ba=bab(ab)^{\sharp}a$.
Recall that every reversible ring is Abelian. Consequently,
\[
ba=b\bigl[ab(ab)^{\sharp}\bigr]a=ba\,ab\,(ab)^{\sharp}.
\]
Furthermore,
\begin{align*}
(ba)(ba)^{\sharp}
&=\bigl[(ba)(ba)^{\sharp}\bigr]ab(ab)^{\sharp} \\
&=\bigl[(ba)(ba)^{\sharp}\bigr]ab\,ab\,(ab)^{\sharp}(ab)^{\sharp} \\
&=aba\bigl[(ba)(ba)^{\sharp}\bigr]b\,(ab)^{\sharp}(ab)^{\sharp} \\
&=ab\,ab\,(ab)^{\sharp}(ab)^{\sharp}=ab(ab)^{\sharp}.
\end{align*}
As $ab\in R^{\mathrm{EP}}$, the projection $ab(ab)^{\sharp}$ is self-adjoint, i.e. $\bigl[(ba)(ba)^{\sharp}\bigr]^{*}=(ba)(ba)^{\sharp}$.
Therefore $ba\in R^{\dagger}$ and $(ba)^{\dagger}=(ba)^{\sharp}$.
\end{proof}

\begin{remark}
The converse of Proposition \ref{pro3.13} does not hold in general.
We verify this claim by means of the counterexample given in Remark \ref{rem3.10}.
Let $R=M_{2}(\mathbb{C})$ be equipped with the involution $*$ defined by conjugate transpose.
Since $M_{2}(\mathbb{C})$ is $*$-regular, the condition $ab\in R^{\dagger}\implies ba\in R^{\dagger}$ holds trivially.
Take
\[
a=\begin{pmatrix}0&1\\0&0\end{pmatrix},\quad
b=\begin{pmatrix}1&0\\0&0\end{pmatrix}.
\]
Then $ab=0$, while $ba=\begin{pmatrix}0&1\\0&0\end{pmatrix}\ne0$.
Consequently, $R$ fails to be reversible.
\end{remark}

By examining the proof of Proposition \ref{pro3.13}, we conclude that reversibility of $R$ yields the implication: $ab\in R^{\mathrm{EP}}\implies ba\in R^{\mathrm{EP}}$.
The subsequent theorem establishes the converse statement under the framework of EP valued elements.

\begin{theorem}\label{thm3.15}
Let $R$ be a $\ast$-ring. Then $R$ is reversible
if and only if $ab\in R^{\mathrm{EP}}$ implies $ba\in R^{\mathrm{EP}}$.
\end{theorem}

\begin{proof}
\textit{Sufficiency.} The forward direction follows directly from the proof of Proposition \ref{pro3.13}.

\textit{Necessity.} Suppose $ab=0$. Obviously $ab\in R^{\mathrm{EP}}$. By our hypothesis, $ba\in R^{\mathrm{EP}}$.
We then obtain
\[
ba=(ba)^{2}(ba)^{\sharp}=0.
\]
Accordingly, $R$ satisfies the definition of a reversible ring.
\end{proof}

Combining Corollary~4.8, Proposition~4.9 and Lemma~4.12, we obtain the following result.

\begin{corollary}\label{cor3.16}
Let $R$ be a $\ast$-ring. Then $R$ is $\ast$-reversible if and only if
\[
\begin{cases}
ab\in R^{\dagger} \;\Longrightarrow\; b^{\ast}a\in R^{\dagger},\\[2mm]
R^{\dagger}\subseteq R^{EP}.
\end{cases}
\]
\end{corollary}

Similarly, combining Lemma~4.12, Proposition~4.13 and Theorem~4.15, we obtain the following.

\begin{corollary}\label{cor3.17}
Let $R$ be a $\ast$-ring. Then $R$ is reversible if and only if
\[
\begin{cases}
ab\in R^{\dagger} \;\Longrightarrow\; ba\in R^{\dagger},\\[2mm]
R^{\dagger}\subseteq R^{EP}.
\end{cases}
\]
\end{corollary}

\section*{Acknowledge}
The authors sincerely thank Prof. Junchao Wei for his valuable comments and suggestions,
which greatly improved the presentation of this paper.
This work is supported by the National Natural Science Foundation of China (12471133)
and the Anhui Provincial Department of Education Natural Science Research Project (2025AHGXZK30855).

\end{document}